\documentclass[reqno]{amsart}
\usepackage{amsmath,amsthm,amsfonts,amssymb}
\RequirePackage{graphicx}
\numberwithin{equation}{section}
\theoremstyle{plain}

\newtheorem{theorem}{Theorem}

\newtheorem{proposition}[theorem]{Proposition}

\newtheorem{remark}[theorem]{Remark}

\begin{document}

\title[Asymptotic Structure of Constrained ERGM]{Asymptotic Structure of Constrained Exponential Random Graph Models}

\author{Lingjiong Zhu}
\address
{Department of Mathematics\newline
\indent Florida State University\newline
\indent 1017 Academic Way\newline
\indent Tallahassee, FL-32306\newline
\indent United States of America}
\email{
zhu@math.fsu.edu}

\date{17 January 2017.}


\begin{abstract}
In this paper, we study exponential random graph models subject to certain constraints.
We obtain some general results about the asymptotic structure of the model.
We show that there exists non-trivial regions in the phase plane
where the asymptotic structure is uniform and there also exists non-trivial regions
in the phase plane where the asymptotic structure is non-uniform. We will get more refined
results for the star model and in particular the two-star model for which a sharp transition
from uniform to non-uniform structure, a stationary point and phase transitions will be obtained.
\end{abstract}

\maketitle

\section{Introduction}

Probabilistic ensembles with one or more adjustable parameters
are often used to model complex networks, see 
e.g. Fienberg \cite{FienbergI,FienbergII}, 
Lov\'{a}sz \cite{Lovasz2009}
and Newman \cite{Newman}. 
One of the standard complex network models used very often in social networks, biological networks, the Internet etc.
is the exponential random graph model, originally
studied by Besag \cite{Besag}. 
We refer to Snijders et al. \cite{Snijders}, Rinaldo et al. \cite{Rinaldo} and
Wasserman and Faust \cite{Wasserman} for history and 
a review of recent developments.

Recently, exponential random graph models and its variations have got a lot of attentions
in the literature. The emphasis has been made on the limiting free energy and entropy, phase transitions
and asymptotic structures, see e.g. Chatterjee and Diaconis \cite{ChatterjeeDiaconis}, 
Radin and Yin \cite{Radin}, Radin and Sadun \cite{RadinII}, 
Radin et al. \cite{RadinIII}, Radin and Sadun \cite{RadinIV}, Kenyon et al. \cite{Kenyon},
Yin \cite{Yin}, Yin et al. \cite{YinII}, Aristoff and Zhu \cite{AristoffZhu},
Aristoff and Zhu \cite{AristoffZhuII}.
In this paper, we are interested to study the constrained exponential random graph models
introduced in Kenyon and Yin \cite{KenyonYin}. The directed case was first studied in Aristoff and Zhu \cite{AristoffZhuII}.

Let us first introduce the exponential random graph model. 
Let $\mathcal{G}_{n}$
be the set of all simple (i.e., undirected, without loops or multiple edges) graphs $G_{n}$
on $n$ vertices. For each $G_{n}\in\mathcal{G}_{n}$, define the probability measure
\begin{equation}
\mathbb{P}_{n}(G_{n})=\exp\left\{n^{2}\left(\beta_{1}t(H_{1},G_{n})+\cdots
+\beta_{k}t(H_{k},G_{n})-\psi_{n}(\beta_{1},\ldots,\beta_{k})\right)\right\},
\end{equation}
where $(\beta_{1},\ldots,\beta_{k})$ are parameters, $H_{1},\ldots,H_{k}$ are given finite simple graphs,
$t(H_{j},G_{n})$, $1\leq j\leq k$ are the densities of graph homomorphisms defined as
\begin{equation}
t(H_{j},G_{n})=\frac{|\text{hom}(H_{j},G_{n})|}{|V(G_{n})|^{|V(H_{j})|}},
\qquad 
j=1,2,\ldots,k,
\end{equation}
and $\psi_{n}(\beta_{1},\ldots,\beta_{k})$ is the normalizing constant
\begin{equation}
\psi_{n}(\beta_{1},\ldots,\beta_{k})=
\frac{1}{n^{2}}\log\sum_{G_{n}\in\mathcal{G}_{n}}\exp
\left\{n^{2}(\beta_{1}t(H_{1},G_{n})+\cdots+\beta_{k}t(H_{k},G_{n}))\right\}.
\end{equation}

Consider a simple graph $H$ with number of vertices denoted by $v(H)$ and number of edges 
denoted by $e(H)$. The set of vertices and the set of edges are denoted by
$V(H)$ and $E(H)$ respectively. Let $V(H)=\{1,2,\ldots,k\}$. We also define
\begin{equation}
t(H,h)=\int_{[0,1]^{k}}\prod_{\{i,j\}\in E(H)}h(x_{i},x_{j})dx_{1}\cdots dx_{k},
\end{equation}
where $h:[0,1]^{2}\rightarrow[0,1]$
with $h(x,y)=h(y,x)$ for any $0\leq x,y\leq y\leq 1$ is known as a \emph{graphon}.

For a more detailed introduction and background about the exponential random graph model, 
we refer to Section 2 of Kenyon and Yin \cite{KenyonYin}.

Then, using the large deviation theory for random graphs developed in Chatterjee and Varadhan \cite{ChatterjeeVaradhan},
the limiting free energy for the exponential random graph models was obtained in Chatterjee and Diaconis \cite{ChatterjeeDiaconis}.

\begin{theorem}[Chatterjee and Diaconis \cite{ChatterjeeDiaconis}]
\begin{align}
&\lim_{n\rightarrow\infty}\psi_{n}(\beta_{1},\ldots,\beta_{k})
\\
&=\sup_{h:[0,1]^{2}\rightarrow[0,1],h(x,y)=h(y,x)}\left\{\sum_{i=1}^{k}\beta_{i}t(H_{i},h)
-\frac{1}{2}\iint_{[0,1]^{2}}I(h(x,y))dxdy\right\},
\nonumber
\end{align}
where $I(x):=x\log x+(1-x)\log(1-x)$. In particular, if $H_{1}$ denotes a single edge
and $\beta_{2},\ldots,\beta_{k}\geq 0$,
\begin{equation}\label{notDistinguish}
\lim_{n\rightarrow\infty}\psi_{n}(\beta_{1},\ldots,\beta_{k})
=\sup_{0\leq x\leq 1}\left\{\beta_{1}x+\sum_{i=2}^{k}\beta_{i}x^{e(H_{i})}-\frac{1}{2}I(x)\right\}.
\end{equation}
\end{theorem}

The equation \eqref{notDistinguish} implies that
when $\beta_{2},\ldots,\beta_{k}\geq 0$, the limiting free energy
$\psi_{n}(\beta_{1},\ldots,\beta_{k})$ does not distinguish
what subgraphs $H_{i}$ are chosen as long as they share
the same $e(H_{i})$. Moreover, since the optimizing graphon is constant,
a typical graph behaves like an Erd\H{o}s-R\'{e}nyi. This suggests that
sometimes subgraph densities cannot be tuned and the exponential random graph model
may not capture all desriable features of the networks in the applications.
This provides the motivation to study variants of the exponential random graph model,
where some subgraph density is controlled. See Kenyon and Yin \cite{KenyonYin}
for more background and discussions on this.

A natural question is what an exponential random graph will look like if it is subject
to certain constraints? For example, what if it is given that the edge density of the graph
is close to $\frac{1}{2}$? What is the asymptotic structure like for the constrained exponential
random graph models? Do we still have phase transition pheonomena as in the classical
exponential random graph models?

In Kenyon and Yin \cite{KenyonYin}, they introduced a constrained exponential random graph model
subject to the edge density of the graph, which will be the focus of this paper.
Let us consider a constrained exponential random graph model with edge density fixed as $0\leq\epsilon\leq 1$.
The conditional normalization constant $\psi_{n,\delta}(\epsilon,\beta_{2},\ldots,\beta_{k})$ is defined as
\begin{equation}\label{MainModelI}
\psi_{n,\delta}(\epsilon,\beta_{2},\ldots,\beta_{k})
=\frac{1}{n^{2}}\log\sum_{G_{n}\in\mathcal{G}_{n}:|e(G_{n})-\epsilon|<\delta}
\exp\left\{n^{2}\sum_{j=2}^{k}\beta_{j}t(H_{j},G_{n})\right\},
\end{equation}
where $H_{j}$, $2\leq j\leq k$ are given simple finite graphs
and the corresponding conditional probability measure is given by
\begin{equation}\label{MainModelII}
\mathbb{P}_{n,\delta}^{\epsilon}(G_{n})
=\exp\left\{-n^{2}\psi_{n,\delta}(\epsilon,\beta_{2},\ldots,\beta_{k})+n^{2}\sum_{j=2}^{k}\beta_{j}t(H_{j},G_{n})\right\}
1_{|e(G_{n})-\epsilon|<\delta}.
\end{equation}
We shrink the interval around $\epsilon$ by letting $\delta$ go to zero:
\begin{equation}
\psi(\epsilon,\beta_{2},\ldots,\beta_{k}):=\lim_{\delta\rightarrow 0}\lim_{n\rightarrow\infty}
\psi_{n,\delta}(\epsilon,\beta_{2},\ldots,\beta_{k}).
\end{equation}

As a result of the large deviations for random graphs \cite{ChatterjeeVaradhan} and Varadhan's lemma
from large deviation theory, we have the following result.

\begin{theorem}[Kenyon and Yin \cite{KenyonYin}]
\begin{align}
&\psi(\epsilon,\beta_{2},\ldots,\beta_{k})
\\
&=\sup_{\substack{h:[0,1]^{2}\rightarrow[0,1],h(x,y)=h(y,x)
\\
\iint_{[0,1]^{2}}h(x,y)dxdy=\epsilon}}\left\{\sum_{j=2}^{k}\beta_{j}t(H_{j},h)-\frac{1}{2}
\iint_{[0,1]^{2}}I(h(x,y))dxdy\right\},
\nonumber
\end{align}
where $I(x)=x\log x+(1-x)\log(1-x)$.
\end{theorem}

As in Kenyon and Yin \cite{KenyonYin}, in our paper, we only
concentrate on the case when $k=2$, i.e., $\beta_{3}=\beta_{4}=\cdots=\beta_{k}=0$,
\begin{equation}\label{MainEqn}
\psi(\epsilon,\beta_{2})=
\sup_{\substack{h:[0,1]^{2}\rightarrow[0,1],h(x,y)=h(y,x)
\\
\iint_{[0,1]^{2}}h(x,y)dxdy=\epsilon}}\left\{\beta_{2}t(H_{2},h)-\frac{1}{2}
\iint_{[0,1]^{2}}I(h(x,y))dxdy\right\}.
\end{equation}
When $H_{2}$ is a triangle, we will call it an edge-triangle model or triangle model.
and when $H_{2}$ is a star, we will call it an edge-star model or star model.

Kenyon and Yin \cite{KenyonYin} mainly considered the \emph{repulsive} regime, i.e., $\beta_{2}<0$. 
They proved that for edge-triangle exponential random graph model, for fixed edge density $\epsilon$,
$\psi^{\epsilon,\beta_{2}}$ is not analytic at at least one value of $\beta_{2}$ when $\beta_{2}$
varies from $0$ to $-\infty$. The same result holds if we replace triangle by a general simple graph
with chromatic number at least $3$. Again for the edge-triangle model, for the special case
when we fix $\epsilon=1/2$, Kenyon and Yin \cite{KenyonYin} showed
that $\psi^{\epsilon,\beta_{2}}$ is analytic everywhere except at one point where the partial derivative
$\frac{\partial}{\partial\beta_{2}}\psi^{\epsilon,\beta_{2}}$ displays a jump discontinuity.

In this paper, we study both the \emph{repulsive}
and \emph{attractive} regimes, with an emphasis on the \emph{attractive} regime, i.e., $\beta_{2}>0$. 

Before we proceed, let us mention
an alternative to exponential random graph models 
that was introduced by Radin and Sadun~\cite{RadinII}, 
where instead of using parameters to control subgraph counts, 
the subgraph densities are controlled directly; 
see also Radin et al.~\cite{RadinIII}, Radin and Sadun~\cite{RadinIV}
and Kenyon et al.~\cite{Kenyon}. For example, we can fix the edge density
and the density of a given simple finite graph $H$ and study the entropy
\begin{align}\label{MicroModel}
\psi(\epsilon,\tau)&:=
-\lim_{\delta\rightarrow 0}\lim_{n\rightarrow\infty}
\frac{1}{n^2}\log\mathbb{P}\left(e(G_{n}) \in (\epsilon-\delta,\epsilon+\delta), \,
t(H,G_{n}) \in (\tau-\delta,\tau+\delta)\right)
\\
&=\sup_{\substack{h:[0,1]^{2}\rightarrow[0,1],h(x,y)=h(y,x)
\\
\iint_{[0,1]^{2}}h(x,y)dxdy=\epsilon, t(H,h)=\tau}}
\left\{-\frac{1}{2}\iint_{[0,1]^{2}}I_{0}(h(x,y))dxdy\right\},
\nonumber
\end{align}
where $\mathbb{P}$ is the uniform probability measure, i.e., Erd\H{o}s-Reny\'{i}
with probability of forming an edge being $\frac{1}{2}$. In \eqref{MicroModel}, $I_{0}(x):=x\log x+(1-x)\log(1-x)+\log 2$.
In the language of statistical mechanics, this is the \emph{micro-canonical} model.
The classical exponential random graph model is the \emph{grand-canonical} model
and the constrained exponential random graph model is the \emph{canonical} model.
There are interesting connections between these three models. Indeed, we'll see later
in this paper that the previous known results about grand-canoncial model can help us
to study the canonical model. Kenyon and Yin \cite{Kenyon} also used the known results
about micro-canonical model to study the canonical model. The interplays and connections
between these three models are worth further investigations in the future.

Before we proceed, we need to review some results from the classical exponential random graph models
and some notations that will be used later in this paper. For the classical exponential random graph models
with $k=2$, the phase transition is well understood for $\beta_{2}$ non-negative and in general
for $p$-star model. The key is the following.

\begin{theorem}[Radin and Yin~\cite{Radin}]
Consider the function
\begin{equation}
\ell(x) := \beta_1 x + \beta_2 x^p - x \log x - (1-x) \log (1-x),\qquad 0\leq x\leq 1.
\end{equation}
For each $(\beta_1,\beta_2)$ 
the function $\ell$ has either one or two local maximizers.
There is a curve $\beta_2 = q(\beta_1)$, $\beta_1 \le \beta_1^c$, with the endpoint
\begin{equation*}
(\beta_{1}^{c},\beta_{2}^{c})=\left(\log(p-1) - \frac{p}{p-1},\frac{p^{p-1}}{(p-1)^p}\right),
\end{equation*}
such that off the curve and at the endpoint, $\ell$ has a 
unique global maximizer, while on the curve away from the 
endpoint, $\ell$ has two global 
maximizers $0 < x_1 < x_2 < 1$. 
The curve $q$ is continuous, decreasing and is called the \emph{phase transition curve}.
\end{theorem}

It was further proved in Aristoff and Zhu \cite{AristoffZhu}
that the phase transition curve $q$ is convex, and analytic 
for $\beta_1 < \beta_1^c$.

Constrained exponential random graph model has been studied in Aristoff and Zhu \cite{AristoffZhuII}
for the edge-star model when the graph is directed. They proved that there exists a U-shaped region
in the phase plane such that the asymptotic structure is uniform outside of this U-shaped region
and is non-uniform otherwise. Here, and for the rest of the paper, ``uniform'' (resp. ``non-uniform'')
means the optimizing graphon in the variational problem that appears 
in the formula for the limiting free energy is a constant function (resp. a non-constant function).
For our purpose, it suffices to quote the following theorem
which will be used later in the proof of Proposition \ref{PositiveUniform}.

\begin{theorem}[Aristoff and Zhu \cite{AristoffZhuII}]\label{UThm}
Consider the optimization problem
\begin{equation}
\psi(\epsilon,\beta_{2})=\sup_{g:[0,1]\rightarrow[0,1],\int_{0}^{1}g(x)dx=\epsilon}
\left\{\beta_{2}\int_{0}^{1}g(x)^{p}dx-\int_{0}^{1}I(g(x))dx\right\}.
\end{equation}
There is a U-shaped region 
\begin{equation*}
U_{\epsilon} = \{(\epsilon,\beta_2)\,:\,x_1 < \epsilon < x_2,\,\beta_2 > \beta_2^c\}
\end{equation*}
whose closure has lowest point 
\begin{equation*}
 \left(\epsilon^c,\beta_2^c\right) = \left(\frac{p-1}{p}, \frac{p^{p-1}}{(p-1)^p}\right)
\end{equation*} 
The optimizer is uniform, i.e., $g(x)\equiv\epsilon$ if $(\epsilon,\beta_{2})\in U_{\epsilon}^{c}$
and the optimizer is given by (unique up to permutation)
\begin{equation}
g(x)=
\begin{cases}
x_{1} &\text{if $0<x<\frac{x_2-\epsilon}{x_{2}-x_{1}}$}
\\
x_{2} &\text{if $\frac{x_2-\epsilon}{x_{2}-x_{1}}<x<1$}
\end{cases},
\end{equation}
if $(\epsilon,\beta_{2})\in U_{\epsilon}$, 
where $0 < x_1 < x_2 <1$ are the global maximizers of $\ell$ at 
the point $(q^{-1}(\beta_2), \beta_2)$ on the phase transition curve. 
\end{theorem}
Note that the optimizing graphon in the variational problem gives us the asymptotic structure
of large graphs. Intuitively, if the optimizer is uniform, the typical graph behaves like an Erd\H{o}s-R\'{e}nyi graph
with the edge connection probability given by the unique optimizer; 
and if the optimizer is bi-podal or multi-podal, then the typical graph behaves like a stochastic block model.

The paper is devoted to study the constrained exponential random graph model. 
In Section \ref{UniformSection}, we will give some very general results
on the uniform and non-uniform structures for the constrained exponential random graph models
for both the attractive regime and the repulsive regime. 
Let $\beta_{2}$ be the parameter associated with the density of
a subgraph $H$, and $\epsilon$ be the fixed edge density. When $\beta_{2}$ is close to zero,
in either attractive or repulsive regime, we will show that the optimal graphon
is uniform and when $\beta_{2}$ is sufficiently large, the optimal graphon will not
be uniform. This is proved and estimates are computed for critical values of the parameters.
In Section \ref{EdgeStarSection}, 
further properties will be studied for the edge-star model, including
when the asymptotic structure is uniform and when the asymptotic structure is multi-podal.
When the underlying graph $H$ is a two-star, more refined results will be given
in Section \ref{TwoStarSection}, including a sharp transition along the line $\epsilon=1/2$,
a stationary point and phase transitions. We conclude the paper with summary and open questions
in Section \ref{SummarySection}.

\section{Uniform and Non-Uniform Structures}\label{UniformSection}

In this section, we study the asymptotic structure of the constrained exponential random graph model
defined in \eqref{MainModelI} and \eqref{MainModelII}.
In particular, we are interested to study when the optimizing graphon in \eqref{MainEqn} is uniform
and when it is not.
When $\beta_{2}\geq 0$, the model favors more subgraph $H$ and the opposite
is true when $\beta_{2}\leq 0$. Consequently, when $\beta_{2}\geq 0$, it is called
the \emph{attractive} regime and when $\beta_{2}<0$, it is called
the \emph{repulsive} regime. We first present some general results
about the asymptotic structure in the attractive regime. Then we will discuss
some general results for the repulsive regime.

\subsection{Attractive Regime}

\begin{proposition}\label{PositiveUniform}
Consider a simple graph $H$ and the conditional exponential random graph model
defined in \eqref{MainModelI} and \eqref{MainModelII}. 
There exists a $U$-shaped region defined in Theorem \ref{UThm} such that the optimizing graphon
in \eqref{MainEqn} is uniform
if $(\epsilon,2\beta_{2})$ is outside of this $U$-shaped region and $\beta_{2}\geq 0$.
\end{proposition}

\begin{proof}
For $\beta_{2}\geq 0$, by generalized H\"{o}lder's inequality and equation \eqref{MainEqn},
\begin{align}
&\psi(\epsilon,\beta_{2})
\\
&=\sup_{\substack{\int_{0}^{1}\int_{0}^{1}h(x,y)dxdy=\epsilon
\\
h(x,y)=h(y,x)}}\left\{\beta_{2}t(H,h)-\frac{1}{2}I(h)\right\}
\nonumber
\\
&\leq\sup_{\substack{\int_{0}^{1}\int_{0}^{1}h(x,y)dxdy=\epsilon
\\
h(x,y)=h(y,x)}}\left\{\beta_{2}\int_{0}^{1}\int_{0}^{1}h(x,y)^{e(H)}dxdy
-\frac{1}{2}\int_{0}^{1}\int_{0}^{1}I(h(x,y))dxdy\right\}
\nonumber
\\
&=\sup_{\int_{0}^{1}\int_{0}^{1}h(x,y)dxdy=\epsilon}\left\{2\beta_{2}\iint_{0<x<y<1}h(x,y)^{e(H)}dxdy
-\iint_{0<x<y<1}I(h(x,y))dxdy\right\}.
\nonumber
\end{align}
We can write down the Euler-Lagrange equation and follow the same arguments as
in \cite{AristoffZhuII} to show that for $(\epsilon,2\beta_{2})$ outside of a $U$-shaped region,
\begin{align}
&\sup_{\int_{0}^{1}\int_{0}^{1}h(x,y)dxdy=\epsilon}\left\{2\beta_{2}\iint_{0<x<y<1}h(x,y)^{e(H)}dxdy
-\iint_{0<x<y<1}I(h(x,y))dxdy\right\}
\\
&\qquad\qquad\qquad\qquad
=\beta_{2}\epsilon^{e(H)}-\frac{1}{2}I(\epsilon)
\nonumber.
\end{align}
On the other hand, it's clear that $\psi(\epsilon,\beta_{2})\geq\beta_{2}\epsilon^{e(H)}-\frac{1}{2}I(\epsilon)$, 
which concludes the proof.
\end{proof}

Proposition \ref{PositiveUniform} shows that outside a $U$-shaped region in the attractive regime,
the optimizing graphon is uniform, that is, the typical graph behaves
like an Erd\H{o}s-R\'{e}nyi graph.

It is then natural to study the optimizing graphon inside the $U$-shaped region. 
We are able to obtain some partial results here.
Note that, for any $\epsilon\in(0,1)$ so that $(\epsilon,\beta_{2})$ is inside the $U$-shaped region
for any sufficiently large $\beta_{2}$.
We will indeed show later that for large finite $\beta_{2}$, 
the optimizing graphon is not uniform.

First, let us study the limiting behavior as $\beta_{2}\rightarrow\infty$.
When $H$ is a two-star, it is known that for fixed edge density $\epsilon$, the maximal possible
two-star density is known to be, see e.g. \cite{Ahlswede}
\begin{equation}
s(\epsilon)=
\begin{cases}
2\epsilon+(1-\epsilon)^{3/2}-1 &\text{$0\leq\epsilon\leq\frac{1}{2}$},
\\
\epsilon^{3/2} &\text{$\frac{1}{2}\leq\epsilon\leq 1$}.
\end{cases}
\end{equation}
And the maximizer is given by an $h$-clique for $\frac{1}{2}\leq\epsilon\leq 1$
\begin{equation}
h_{c}(x,y)=
\begin{cases}
1 &\text{if $x<\sqrt{\epsilon}$ and $y<\sqrt{\epsilon}$}
\\
0 &\text{otherwise}
\end{cases},
\end{equation}
and the maximizer is given by an $h$-anticlique for $0\leq\epsilon\leq\frac{1}{2}$
\begin{equation}
h_{a}(x,y)=
\begin{cases}
0 &\text{if $x>1-\sqrt{1-\epsilon}$ and $y>1-\sqrt{1-\epsilon}$}
\\
1 &\text{otherwise}
\end{cases}.
\end{equation}

For the triangle model, i.e., when $H$ is a triangle, given the edge density $\epsilon$,
the maximal possible triangle density is $t(\epsilon)=\epsilon^{3/2}$,
see \cite{Pikhurko} and the references therein. It is easy to
check that the clique 
\begin{equation}
h_{c}(x,y)=
\begin{cases}
1 &\text{if $x<\sqrt{\epsilon}$ and $y<\sqrt{\epsilon}$}
\\
0 &\text{otherwise}
\end{cases}
\end{equation}
gives the optimizer.

\begin{proposition}\label{Limit}
\begin{equation}
\lim_{\beta_{2}\rightarrow\infty}\frac{1}{\beta_{2}}\psi(\epsilon,\beta_{2})
=\sup_{\iint_{[0,1]^{2}}h(x,y)dxdy=\epsilon,h(x,y)=h(y,x)}t(h,H).
\end{equation}
In particular, for the two-star model
\begin{equation}
\lim_{\beta_{2}\rightarrow\infty}\frac{1}{\beta_{2}}\psi(\epsilon,\beta_{2})
=
\begin{cases}
2\epsilon+(1-\epsilon)^{3/2}-1 &\text{$0\leq\epsilon\leq\frac{1}{2}$},
\\
\epsilon^{3/2} &\text{$\frac{1}{2}\leq\epsilon\leq 1$}.
\end{cases}
\end{equation}
and for the triangle model
\begin{equation}
\lim_{\beta_{2}\rightarrow\infty}\frac{1}{\beta_{2}}\psi(\epsilon,\beta_{2})
=\epsilon^{3/2}.
\end{equation}
\end{proposition}

\begin{proof}
It is easy to check that $I(x)=x\log x+(1-x)\log(1-x)$ is decreasing on $[0,1/2]$ and increasing
on $[1/2,1]$ and $I(0)=I(1)=0$, $I(1/2)=-\log 2$. Therefore, for any $\beta_{2}>0$,
\begin{align}
&\sup_{\iint_{[0,1]^{2}}h(x,y)dxdy=\epsilon,h(x,y)=h(y,x)}\beta_{2}\cdot t(h,H)
\\
&\qquad\qquad\qquad
\leq\psi(\epsilon,\beta_{2})\leq\sup_{\iint_{[0,1]^{2}}h(x,y)dxdy=\epsilon,h(x,y)=h(y,x)}\beta_{2}\cdot t(h,H)+\frac{1}{2}\log 2.
\nonumber
\end{align}
\end{proof}

\begin{remark}
Let $\mathcal{H}$ be the set of optimizers of $t(h,H)$ given edge density $\epsilon$. 
Let $h_{\epsilon,\beta_{2}}$ be an optimizing graphon for $(\epsilon,\beta_{2})$. 
Then, the distance between $h_{\epsilon,\beta_{2}}$ and $\mathcal{H}$ goes to zero
as $\beta_{2}\rightarrow\infty$ in the cut metric. To see this, suppose not, 
since the space of reduced graphons is compact, see e.g. \cite{Lov}, there must be an accumulation point $h_{\epsilon}
\notin\mathcal{H}$ for the sequence $(h_{\epsilon,\beta_{2}})_{\beta_{2}}$. 
There exists a subsequence $h_{\epsilon,\tilde{\beta}_{2}}\rightarrow h_{\epsilon}$ in the cut metric
which implies that $t(h_{\epsilon,\tilde{\beta}_{2}})\rightarrow t(h_{\epsilon})$ as $\tilde{\beta}_{2}\rightarrow\infty$.
By Proposition \ref{Limit}, it is easy to see that $t(h_{\epsilon})=\sup_{\iint_{[0,1]^{2}}h(x,y)dxdy=\epsilon,h(x,y)=h(y,x)}t(h,H)$.
Therefore, we must have $h_{\epsilon}\in\mathcal{H}$ which is a contradiction.
\end{remark}

Recall that given the edge density $\epsilon$, the maximal possible triangle density
is $\epsilon^{3/2}$ achieved by the clique 
$h_{c}(x,y)=1$ if $0<x,y<\sqrt{\epsilon}$ and $h_{c}(x,y)=0$ otherwise.
Thus, it is easy to compute that
\begin{align}
&\beta_{2}\iiint_{[0,1]^{3}}h_{c}(x,y)h_{c}(y,z)h_{c}(z,x)dxdydz
-\frac{1}{2}\iint_{[0,1]^{2}}I(h_{c}(x,y))dxdy
\\
&\qquad\qquad\qquad\qquad\qquad
-\beta_{2}\epsilon^{3}+\frac{1}{2}I(\epsilon)
\nonumber
\\
&=\beta_{2}(\epsilon^{3/2}-\epsilon^{3})+\frac{1}{2}[\epsilon\log\epsilon+(1-\epsilon)\log(1-\epsilon)].
\nonumber
\end{align}
Hence, the optimizer for the triangle model is not uniform if
\begin{equation}
\beta_{2}>\frac{-\frac{1}{2}[\epsilon\log\epsilon+(1-\epsilon)\log(1-\epsilon)]}{\epsilon^{3/2}-\epsilon^{3}}.
\end{equation}

In general, we have the following result.

\begin{proposition}\label{VE}
Let $H$ be a simple graph with number of vertices and edges
denoted by $v(H)$ and $e(H)$ respectively such that $e(H)>v(H)/2$.
Then, the optimizing graphon in \eqref{MainEqn} is non-uniform if
\begin{equation}
\beta_{2}>\frac{-\frac{1}{2}[\epsilon\log\epsilon+(1-\epsilon)\log(1-\epsilon)]}{\epsilon^{v(H)/2}-\epsilon^{e(H)}}.
\end{equation}
\end{proposition}

\begin{remark}
Recall that for the classical exponential random graph model, the optimizing graphon is uniform
for any $\beta_{2}>0$, see Chatterjee and Diaconis \cite{ChatterjeeDiaconis}. Proposition \ref{VE}
demonstrates that this is not the case for constrained exponential random graph models.
Indeed, for sufficiently large $\beta_{2}$, you always have non-uniform structure.
It would be then very interesting to study for finite large $\beta_{2}$, the exact structure
for the optimizing graphon.
>From the discussions above Proposition \ref{VE} and also the proof given below, it is
natural to conjecture that for large finite $\beta_{2}$, the optimizing graphon
is a clique with size determined by the edge density. It remains an open problem
to prove or disprove this.
\end{remark}

\begin{proof}[Proof of Proposition \ref{VE}]
We define the clique 
$h_{c}(x,y)=1$ if $0<x,y<\sqrt{\epsilon}$ and $h_{c}(x,y)=0$ otherwise.
Thus, it is easy to compute that
\begin{align}
&\beta_{2}\int_{[0,1]^{v(H)}}\prod_{\{i,j\}\in E(H)}h_{c}(x_{i},x_{j})dx_{1}\cdots dx_{v(H)}
-\frac{1}{2}\iint_{[0,1]^{2}}I(h_{c}(x,y))dxdy
\\
&\qquad\qquad\qquad\qquad\qquad
-\beta_{2}\epsilon^{e(H)}+\frac{1}{2}I(\epsilon)
\nonumber
\\
&=\beta_{2}(\epsilon^{v(H)/2}-\epsilon^{e(H)})+\frac{1}{2}[\epsilon\log\epsilon+(1-\epsilon)\log(1-\epsilon)].
\nonumber
\end{align}
Hence, the optimizer is not uniform if
\begin{equation}
\beta_{2}>\frac{-\frac{1}{2}[\epsilon\log\epsilon+(1-\epsilon)\log(1-\epsilon)]}{\epsilon^{v(H)/2}-\epsilon^{e(H)}}.
\end{equation}
\end{proof}

\subsection{Repulsive Regime}

For the \emph{repulsive} regime, i.e., $\beta_{2}\leq 0$, Kenyon and Yin \cite{KenyonYin}
showed non-analyticity as $\beta_{2}$ varies from $0$ to $-\infty$ when $H$ is a general simple graph
with chromatic number at least $3$. This implicitly tells us that the optimizing graphon
cannot be uniform everywhere for $\beta_{2}\leq 0$. 
Furthermore, for the edge-triangle model along $\epsilon=1/2$, using the micro model results
by Radin and Sadun \cite{RadinIV}, it was pointed out in Kenyon and Yin \cite{KenyonYin}
that for negative $\beta_{2}$,
\begin{align}
&\psi\left(\frac{1}{2},\beta_{2}\right)
\\
&=\sup_{0\leq\tau\leq\frac{1}{8}}\sup_{
\substack{
\iint_{[0,1]^{2}}h(x,y)dxdy=\frac{1}{2}
\\
\iiint_{[0,1]^{3}}h(x,y)h(y,z)h(z,x)dxdydz=\tau
\\
h(x,y)=h(y,x)}}\left\{\beta_{2}\tau-\frac{1}{2}\iint_{[0,1]^{2}}I(h(x,y))dxdy\right\}
\nonumber
\\
&=\sup_{0\leq\tau\leq\frac{1}{8}}\left\{\beta_{2}\tau-\frac{1}{2}\iint_{[0,1]^{2}}I(h_{\tau}(x,y))dxdy\right\},
\nonumber
\end{align}
where
\begin{equation}
h_{\tau}(x,y)=
\begin{cases}
\frac{1}{2}+(\frac{1}{8}-\tau)^{\frac{1}{3}} &\text{if $x<\frac{1}{2}<y$ or $x>\frac{1}{2}>y$}
\\
\frac{1}{2}-(\frac{1}{8}-\tau)^{\frac{1}{3}} &\text{if $x,y<\frac{1}{2}$ or $x,y>\frac{1}{2}$}
\end{cases}
\end{equation}
is the optimizer for the micro model and thus $h_{\tau(\beta_{2})}$ is the optimizer for the canonical
model where
\begin{equation}
\tau(\beta_{2}):=\arg\max
\left\{\beta_{2}\tau-\frac{1}{2}I\left(\frac{1}{2}+\left(\frac{1}{8}-\tau\right)^{1/3}\right)\right\}.
\end{equation}
It is easy to verify that there exists some $\beta_{2}^{c}<0$
so that $\tau(\beta_{2})=\frac{1}{8}$ if $\beta_{2}\geq\beta_{2}^{c}$
and $\tau(\beta_{2})<\frac{1}{8}$ otherwise.
This tells us that along $\epsilon=1/2$ for the edge-triangle model,
the optimizing graphon is uniform for $\beta_{2}^{c}\leq\beta_{2}\leq 0$
and non-uniform for $\beta_{2}<\beta_{2}^{c}$. 

For general $\epsilon\leq 1/2$, 
\begin{equation}
\psi\left(\frac{1}{2},\beta_{2}\right)
\geq\sup_{0\leq\tau\leq\frac{1}{8}}\left\{\beta_{2}\tau-\frac{1}{2}\iint_{[0,1]^{2}}I(h_{\tau}(x,y))dxdy\right\},
\end{equation} 
where 
\begin{equation}
h_{\tau}(x,y)=
\begin{cases}
\epsilon+(\epsilon^{3}-\tau)^{1/3} &\text{if $x<\frac{1}{2}<y$ or $y<\frac{1}{2}<x$}
\\
\epsilon-(\epsilon^{3}-\tau)^{1/3} &\text{otherwise}
\end{cases}
\end{equation}
is a local optimizer for the micro model with $\tau$ being the triangle density (see Radin and Sadun \cite{RadinII}).
By the same analysis as before,
we can see that the optimizing graphon is non-uniform for $\beta_{2}<\beta_{2}^{c}$, where $\beta_{2}^{c}<0$ is a critical value. 
If indeed $h_{\tau}$ is a global optimizer, then the optimizing graphon is uniform for $\beta_{2}^{c}\leq\beta_{2}\leq 0$.

\begin{proposition}\label{NegativeUniform}
For $\beta_{2}<0$ and $|\beta_{2}|e(H)(e(H)-1)<2$, the optimizing graphon in \eqref{MainEqn}
is uniform for any edge density $\epsilon$.
\end{proposition}

\begin{proof}
For $\beta_{2}<0$ and $|\beta_{2}|e(H)(e(H)-1)<2$, Chatterjee and Diaconis \cite{ChatterjeeDiaconis}
proved that the optimizing graphon for the macro model is uniform, i.e., 
\begin{equation}
\psi(\beta_{1},\beta_{2})=\sup_{0\leq\epsilon\leq 1}
\left\{\beta_{1}\epsilon+\beta_{2}\epsilon^{p}-\frac{1}{2}I(\epsilon)\right\}.
\end{equation}
On the other hand,
\begin{align}
\psi(\beta_{1},\beta_{2})
&=\sup_{h(x,y)=h(y,x)}\left\{\beta_{1}e(h)+\beta_{2}t(h,H)-\frac{1}{2}\iint_{[0,1]^{2}}I(h(x,y))dxdy\right\}
\\
&=\sup_{0\leq\epsilon\leq 1}
\sup_{\substack{h(x,y)=h(y,x)
\\
e(h)=\epsilon}}\left\{\beta_{1}e(h)+\beta_{2}t(h,H)-\frac{1}{2}\iint_{[0,1]^{2}}I(h(x,y))dxdy\right\}
\nonumber
\\
&=\sup_{0\leq\epsilon\leq 1}\left\{\beta_{1}\epsilon+
\sup_{\substack{h(x,y)=h(y,x)
\\
e(h)=\epsilon}}\left\{\beta_{2}t(h,H)-\frac{1}{2}\iint_{[0,1]^{2}}I(h(x,y))dxdy\right\}\right\}
\nonumber
\\
&=\sup_{0\leq\epsilon\leq 1}
\left\{\beta_{1}\epsilon+\beta_{2}\epsilon^{p}-\frac{1}{2}I(\epsilon)\right\}
\nonumber
\\
&=\beta_{1}\epsilon^{\ast}+\beta_{2}(\epsilon^{\ast})^{p}-\frac{1}{2}I(\epsilon^{\ast}),
\nonumber
\end{align}
where $\epsilon^{\ast}$ is a maximizer of $\beta_{1}\epsilon+\beta_{2}\epsilon^{p}-\frac{1}{2}I(\epsilon)$.
Hence, we must have
\begin{equation}
\sup_{\substack{h(x,y)=h(y,x)
\\
e(h)=\epsilon^{\ast}}}\left\{\beta_{2}t(h,H)-\frac{1}{2}\iint_{[0,1]^{2}}I(h(x,y))dxdy\right\}
\leq\beta_{2}(\epsilon^{\ast})^{p}-\frac{1}{2}I(\epsilon^{\ast}),
\end{equation}
Therefore, for $(\epsilon^{\ast},\beta_{2})$, the optimizing graphon for the canonical model is uniform.
Notice that the choice of $\beta_{1}$ is arbitrary, thus, for any 
$(\epsilon,\beta_{2})$, the optimizing graphon for the canonical model is uniform if
\begin{equation}
\epsilon\in\bigcup_{\beta_{1}\in\mathbb{R}}\arg\max\left\{\beta_{1}x+\beta_{2}x^{p}-\frac{1}{2}I(x)\right\}.
\end{equation}
For any $\beta_{2}<0<\frac{p^{p-1}}{2(p-1)^{p}}$, by Proposition 3.2. and its proof in Radin and Yin \cite{Radin}, 
there is a unique
maximizer of $\beta_{1}x+\beta_{2}x^{p}-\frac{1}{2}I(x)$ and it increases
from $0$ to $1$ as $\beta_{1}$ varies from $-\infty$ to $\infty$. Therefore,
\begin{equation}
\bigcup_{\beta_{1}\in\mathbb{R}}\arg\max\left\{\beta_{1}x+\beta_{2}x^{p}-\frac{1}{2}I(x)\right\}=(0,1),
\end{equation}
and the optimizing graphon for the canonical model is uniform for any $\epsilon\in(0,1)$.
\end{proof}

\begin{remark}
For $\beta_{2}\geq 0$, Chatterjee and Diaconis \cite{ChatterjeeDiaconis} proved that
\begin{equation}
\psi(\beta_{1},\beta_{2})=\sup_{0\leq\epsilon\leq 1}\left\{\beta_{1}\epsilon+\beta_{2}\epsilon^{p}-\frac{1}{2}I(\epsilon)\right\}.
\end{equation}
Replacing $\beta_{1}$ and $\beta_{2}$ by $\frac{\beta_{1}}{2}$ and $\frac{\beta_{2}}{2}$ respectively, as
in the discussion in Proposition \ref{NegativeUniform}, for fixed $\beta_{2}$, 
the optimizing graphon is uniform if $\epsilon$ lies in the set
\begin{equation}
\bigcup_{\beta_{1}\in\mathbb{R}}\arg\max\{\beta_{1}x+\beta_{2}x^{p}-I(x)\}.
\end{equation}
>From the properties of $\beta_{1}x+\beta_{2}x^{p}-I(x)$ studied in \cite{Radin}, \cite{AristoffZhu}, \cite{AristoffZhuII},
for $\beta_{2}\leq\frac{p^{p-1}}{(p-1)^{p}}$, as $\beta_{1}$ increases from $-\infty$ to $+\infty$, the maximizer of $\beta_{1}x+\beta_{2}x^{p}-I(x)$
increases from $0$ to $1$, while for $\beta_{2}>\frac{p^{p-1}}{(p-1)^{p}}$, as $\beta_{1}$ increases from $-\infty$
to $q^{-1}(\beta_{2})$, the maximizer of $\beta_{1}x+\beta_{2}x^{p}-I(x)$ increases from $0$ to $x_{1}$,
and as $\beta_{1}$ increases from $q^{-1}(\beta_{2})$ to $\infty$, the maximizer of $\beta_{1}x+\beta_{2}x^{p}-I(x)$
increases from $x_{2}$ to $1$, where $0<x_{1}<x_{2}<1$ are the two maximizers of $\beta_{1}x+\beta_{2}x^{p}-I(x)$
for $\beta_{1}=q^{-1}(\beta_{2})$. Hence, we proved that the optimizing graphon in the canonical model
for $\beta_{2}\geq 0$ is uniform if $(\epsilon,2\beta_{2})$ is outside of the $U$-shaped region
as in Proposition \ref{PositiveUniform}.
\end{remark}

For a general simple graph $H$ satisfying some mild conditions, we proved
in Proposition \ref{PositiveUniform} and Proposition \ref{VE} that
there exists a region in the phase plane in which the optimizing graphon
is uniform and there also exists a region in which the optimizing graphon
is not uniform. In general, it seems to be difficult to give a sharp boundary
across which the optimizing graphon changes from uniform to non-uniform except
for some very special cases, e.g. along the line $\epsilon=1/2$ in Proposition \ref{OneHalfLine}.
In the spirit of Proposition \ref{PositiveUniform}
and Proposition \ref{VE}, a natural question we can ask is for fixed edge density $\epsilon$, 
whether there exists $0<\beta_{c}^{1}<\beta_{c}^{2}<\infty$
such that the optimizing graphon is uniform for $0<\beta_{2}<\beta_{c}^{1}$, non-uniform for
$\beta_{c}^{1}<\beta_{2}<\beta_{c}^{2}$ and uniform again for $\beta_{2}>\beta_{c}^{2}$.
The answer turns out to be negative.

\begin{proposition}
Fix the edge density $\epsilon$. If the optimizing graphon is non-uniform for some $\beta_{2}>0$, 
then it is non-uniform for any $\overline{\beta}_{2}>\beta_{2}$. Similarly, if
the optimizing graphon is non-uniform for some $\beta_{2}<0$, then it is non-uniform for any $\underline{\beta}_{2}<\beta_{2}$.
\end{proposition}

\begin{proof}
With loss of generality, we consider the case $\beta_{2}>0$. There exists a non-uniform graphon $h$ such that
\begin{equation}
\beta_{2}t(h,H)-\frac{1}{2}\iint_{[0,1]^{2}}I(h(x,y))dxdy
>\beta_{2}\epsilon^{e(H)}-\frac{1}{2}I(\epsilon).
\end{equation}
This is equivalent to
\begin{equation}
\beta_{2}(t(h,H)-\epsilon^{e(H)})\iint_{[0,1]^{2}}I(h(x,y))dxdy
>\frac{1}{2}\left[\iint_{[0,1]^{2}}I(h(x,y))dxdy-I(\epsilon)\right].
\end{equation}
By Jensen's inequality, $\iint_{[0,1]^{2}}I(h(x,y))dxdy\geq I(\epsilon)$.
Since $\beta_{2}>0$, we get $t(h,H)>\epsilon^{e(H)}$.
Therefore, for any $\overline{\beta}_{2}>\beta_{2}$, we have
\begin{equation}
\overline{\beta}_{2}(t(h,H)-\epsilon^{e(H)})\iint_{[0,1]^{2}}I(h(x,y))dxdy
>\frac{1}{2}\left[\iint_{[0,1]^{2}}I(h(x,y))dxdy-I(\epsilon)\right].
\end{equation}
Thus, the optimizer cannot be uniform at $\overline{\beta}_{2}$.
\end{proof}

\section{Asymptotic Structure for Edge-Star Model}\label{EdgeStarSection}

In Proposition \ref{PositiveUniform}, we proved uniform structure
of the constrained exponential random graph model for very general simple finite graph $H$.
The results in Proposition \ref{PositiveUniform} are restricted to non-negative $\beta_{2}$.
We will show in the following result that for the edge-star model, the uniform structure
holds for any negative $\beta_{2}$.

\begin{proposition}\label{StarUniform}
When $H$ is a $p$-star, there exists a $U$-shaped region as defined in Theorem \ref{UThm} such that 
the optimizing graphon in \eqref{MainEqn} is uniform
for any $(\epsilon,2\beta_{2})$ outside this $U$-shaped region.
\end{proposition}

\begin{proof}
For the $p$-star model,
\begin{equation}
\psi(\epsilon,\beta_{2})=\sup_{\substack{\int_{0}^{1}\int_{0}^{1}h(x,y)dxdy=\epsilon
\\
h(x,y)=h(y,x)}}
\left\{\beta_{2}\int_{0}^{1}\left(\int_{0}^{1}h(x,y)dy\right)^{p}dx
-\frac{1}{2}\int_{0}^{1}\int_{0}^{1}I(h(x,y))dxdy\right\}.
\end{equation}
Since $x\mapsto I(x)$ is convex, Jensen's inequality implies that
\begin{align}
&\psi(\epsilon,\beta_{2})
\\
&\leq\sup_{\substack{\int_{0}^{1}\int_{0}^{1}h(x,y)dxdy=\epsilon
\\
h(x,y)=h(y,x)}}
\left\{\beta_{2}\int_{0}^{1}\left(\int_{0}^{1}h(x,y)dy\right)^{p}dx
-\frac{1}{2}\int_{0}^{1}I\left(\int_{0}^{1}h(x,y)dy\right)dx\right\}
\nonumber
\\
&\leq\sup_{\int_{0}^{1}\int_{0}^{1}h(x,y)dxdy=\epsilon}
\left\{\beta_{2}\int_{0}^{1}\left(\int_{0}^{1}h(x,y)dy\right)^{p}dx
-\frac{1}{2}\int_{0}^{1}I\left(\int_{0}^{1}h(x,y)dy\right)dx\right\}
\nonumber
\\
&=\frac{1}{2}\sup_{\int_{0}^{1}g(x)dx=\epsilon}
\left\{2\beta_{2}\int_{0}^{1}g(x)^{p}dx
-\int_{0}^{1}I(g(x))dx\right\}.
\nonumber
\end{align}
It was proved in \cite{AristoffZhuII} that for $(\epsilon,2\beta_{2})$ outside of a $U$-shaped region,
the optimal $g$ is uniform, i.e., $g(x)\equiv\epsilon$. 
On the other hand, it's clear that $\psi(\epsilon,\beta_{2})\geq\beta_{2}\epsilon^{p}-\frac{1}{2}I(\epsilon)$.
Therefore, the optimizer is uniform outside of a $U$-shaped region.
\end{proof}

\begin{remark}
For the $p$-star model, for $\beta_{2}\leq 0$, by Jensen's inequality
\begin{align}
\psi(\epsilon,\beta_{2})&\leq\sup_{\substack{\int_{0}^{1}\int_{0}^{1}h(x,y)dxdy=\epsilon
\\
h(x,y)=h(y,x)}}
\left\{\beta_{2}\left(\int_{0}^{1}\int_{0}^{1}h(x,y)dydx\right)^{p}
-\frac{1}{2}I\left(\int_{0}^{1}\int_{0}^{1}h(x,y)dxdy\right)\right\}
\\
&=\beta_{2}\epsilon^{p}-\frac{1}{2}I(\epsilon).\nonumber
\end{align}
Together with Proposition \ref{PositiveUniform}, we recover the conclusion in Proposition \ref{StarUniform}.
\end{remark}

In a very recent paper by Kenyon et al. \cite{Kenyon}, they proved a remarkable
result that for the micro-canonical edge-star model, i.e., the model defined in \eqref{MicroModel}
for $H$ being a $p$-star, the optimizing graphon is always multipodal.
Following their argument, it is easy to see that when $H$ is a $p$-star, for the constrained
exponential random graph model \eqref{MainModelI}, \eqref{MainModelII}, the optimizing graphon
is always multipodal. Unlike the micro-canonical model, the parameter $\beta_{2}$
is given for the constrained exponential random graph model. Therefore, there is a need
to make the parameter $\beta_{2}$ more transparent in the Euler-Lagrange equation etc. which will
be used in the proof of Proposition \ref{OneHalfLine}. 

\begin{proposition}\label{Multi}
When $H$ is a $p$-star in the constrained exponential random graph model \eqref{MainModelI},\eqref{MainModelII}, 
the optimizing graphon in \eqref{MainEqn} is multipodal.
\end{proposition}

\begin{proof}
Let us introduce the Lagrange multiplier $\beta_{1}$ and define
\begin{align}
\Lambda(h)&:=\beta_{2}\int_{0}^{1}\left(\int_{0}^{1}h(x,y)dy\right)^{p}dx
+\beta_{1}\left(\epsilon-\int_{0}^{1}\int_{0}^{1}h(x,y)dxdy\right)
\\
&\qquad\qquad\qquad\qquad
-\frac{1}{2}\int_{0}^{1}\int_{0}^{1}I(h(x,y))dxdy.
\nonumber
\end{align}
Consider symmetric $\eta(x,y)=\eta(y,x)$ and set equal to zero
the derivative with respect to $\varepsilon$
\begin{equation}
\frac{d}{d\varepsilon}\bigg|_{\varepsilon=0}\Lambda(h+\varepsilon\eta)=0.
\end{equation}
Thus, we get
\begin{equation}
2\beta_{1}-\beta_{2}pg^{p-1}(x)-\beta_{2}pg^{p-1}(y)=\log\left(\frac{1-h(x,y)}{h(x,y)}\right),
\end{equation}
where $g(x):=\int_{0}^{1}h(x,y)dy$. Rearranging the equation and integrating over $y$,
\begin{equation}
g(x)=\int_{0}^{1}\frac{dy}{1+e^{2\beta_{1}-\beta_{2}pg^{p-1}(x)-\beta_{2}pg^{p-1}(y)}}.
\end{equation}
The values of $g(x)$ are therefore the roots of
\begin{equation}\label{FDefn}
F(z):=z-\int_{0}^{1}\frac{dy}{1+e^{2\beta_{1}-\beta_{2}pz^{p-1}-\beta_{2}pg^{p-1}(y)}}.
\end{equation}
Following the same arguments in the proof of Theorem 3.4. in Kenyon et al. \cite{Kenyon}, the optimizer is multipodal.
\end{proof}

\section{Two-Star Model}\label{TwoStarSection}

In this section, we study in details the more refined properties when the given graph $H$ is a two-star.
In particular, we will show that $U$-shaped region is not optimal, and will give a sharp result
along the line $\epsilon=1/2$, as well as giving a stationary point. Phase transitions will also be discussed.

Unlike the constrained exponential random graph models for directed graphs, see Aristoff and Zhu \cite{AristoffZhuII},
the $U$-shaped region for undirected graphs is not optimal, in the sense that 
inside the $U$-shaped region, the optimal graphon can still be uniform,
which can be seen from the sharp result along the line $\epsilon=1/2$ in Proposition \ref{OneHalfLine}.
For an illustration, we refer to Figure 1.

\begin{proposition}\label{TwoStar}
When $H$ is a two-star, the optimizing graphon in \eqref{MainEqn} is not uniform if $\beta_{2}>\frac{1}{2\epsilon(1-\epsilon)}$.
\end{proposition}

\begin{proof}
For the two-star model,
\begin{equation}
\psi(\epsilon,\beta_{2})=\sup_{\substack{\int_{0}^{1}\int_{0}^{1}h(x,y)dxdy=\epsilon
\\
h(x,y)=h(y,x)}}
\left\{\beta_{2}\int_{0}^{1}\left(\int_{0}^{1}h(x,y)dy\right)^{2}dx
-\frac{1}{2}\int_{0}^{1}\int_{0}^{1}I(h(x,y))dxdy\right\}.
\end{equation}
Let us define
\begin{equation}
h_{\alpha,\delta,\eta}(x,y)=
\begin{cases}
\epsilon+\delta &\text{if $0<x,y<\alpha$ or $\alpha<x,y<1$}
\\
\epsilon-\eta &\text{otherwise}
\end{cases}.
\end{equation}
To satisfy the constraint $\iint_{[0,1]^{2}}h_{\alpha,\delta,\eta}(x,y)dxdy=\epsilon$, we need to impose the condition
\begin{equation}\label{relation}
[\alpha^{2}+(1-\alpha)^{2}]\delta=2\alpha(1-\alpha)\eta.
\end{equation}
It is straightforward to compute that
\begin{align}
&\int_{0}^{1}\left(\int_{0}^{1}h_{\alpha,\delta,\eta}(x,y)dy\right)^{2}dx-\int_{0}^{1}\left(\int_{0}^{1}\epsilon dy\right)^{2}dx
\\
&=[\epsilon+(1-\alpha)\delta-\alpha\eta]^{2}(1-\alpha)
+[\epsilon+\delta\alpha-(1-\alpha)\eta]^{2}\alpha-\epsilon^{2}
\nonumber
\\
&=[(1-\alpha)\delta-\alpha\eta]^{2}(1-\alpha)
+[\delta\alpha-(1-\alpha)\eta]^{2}\alpha.
\nonumber
\end{align}
Notice the last line above is strictly positive if $\alpha\neq\frac{1}{2}$. Therefore,
for $\beta_{2}$ sufficiently large, $h=h_{\alpha,\delta,\eta}$ is more optimal than $h\equiv\epsilon$
and the optimizer is therefore not uniform.

Indeed, the optimizer is not uniform if
\begin{equation}
\beta_{2}\geq\frac{[\alpha^{2}+(1-\alpha)^{2}]\frac{1}{2}I(\epsilon+\delta)
+2\alpha(1-\alpha)\frac{1}{2}I(\epsilon-\eta)
-\frac{1}{2}I(\epsilon)}{[(1-\alpha)\delta-\alpha\eta]^{2}(1-\alpha)
+[\delta\alpha-(1-\alpha)\eta]^{2}\alpha}.
\end{equation}
For $\delta,\eta$ sufficiently small and use \eqref{relation},
the optimizer is not uniform if
\begin{equation}
\beta_{2}\geq\frac{\frac{1}{4}I''(\epsilon)[\alpha^{2}+(1-\alpha)^{2}]\delta(\delta+\eta)+O(\delta^{3})}{[(1-\alpha)\delta-\alpha\eta]^{2}(1-\alpha)
+[\delta\alpha-(1-\alpha)\eta]^{2}\alpha}.
\end{equation}
Fix $\alpha$ and let $\delta,\eta\rightarrow 0$ and again use \eqref{relation}, the optimizer is not uniform if
\begin{align}
\beta_{2}&>
\frac{\frac{1}{4}I''(\epsilon)[\alpha^{2}+(1-\alpha)^{2}](1+\frac{\alpha^{2}+(1-\alpha)^{2}}{2\alpha(1-\alpha)})}
{(1-\alpha-\frac{\alpha^{2}+(1-\alpha)^{2}}{2(1-\alpha)})^{2}(1-\alpha)
+(\alpha-\frac{\alpha^{2}+(1-\alpha)^{2}}{2\alpha})^{2}\alpha}
\\
&=\frac{1}{4\epsilon(1-\epsilon)}\frac{2[\alpha^{2}+(1-\alpha)^{2}]}{(1-2\alpha)^{2}}
\nonumber
\\
&=\frac{1}{2\epsilon(1-\epsilon)}\frac{1+2\alpha^{2}-2\alpha}{(1-2\alpha)^{2}}.
\nonumber
\end{align}
It is easy to check that the minimum of $\frac{1+2\alpha^{2}-2\alpha}{(1-2\alpha)^{2}}$
is achieved at $\alpha=\{0,1\}$. Therefore, the optimizer is not uniform if $\beta_{2}>\frac{1}{2\epsilon(1-\epsilon)}$.
\end{proof}

\begin{remark}\label{UNotOptimal}
If $H$ is a two-star, by Proposition \ref{TwoStar} and Proposition \ref{VE}, the optimizer is not uniform if
\begin{align}
\beta_{2}&>\frac{1}{2\epsilon}\min\left\{\frac{1}{1-\epsilon},\frac{-\epsilon\log\epsilon-(1-\epsilon)\log(1-\epsilon)}
{\sqrt{\epsilon}-\epsilon}\right\}
\\
&=\frac{1}{2\epsilon^{3/2}(1-\sqrt{\epsilon})}
\min\left\{\frac{\sqrt{\epsilon}}{1+\sqrt{\epsilon}},-\epsilon\log\epsilon-(1-\epsilon)\log(1-\epsilon)\right\}.
\nonumber
\end{align}
It is easy to compute that when $\epsilon$ is close to $1$, $\frac{1}{2\epsilon^{3/2}(1-\sqrt{\epsilon})}\{-\epsilon\log\epsilon-(1-\epsilon)\log(1-\epsilon)\}$ gives
a better lower bound for $\beta_{2}$ and when $\epsilon$ is close to $1/2$,
$\frac{1}{2\epsilon^{3/2}(1-\sqrt{\epsilon})}\frac{\sqrt{\epsilon}}{1+\sqrt{\epsilon}}$
gives a better lower bound for $\beta_{2}$.
We illustrate the lower bounds in Proposition \ref{VE} and Proposition \ref{TwoStar},
and also the $U$-shaped region in Figure 1.
\end{remark}

\begin{figure}
\begin{center}
\hskip-25pt\includegraphics[scale=0.75]{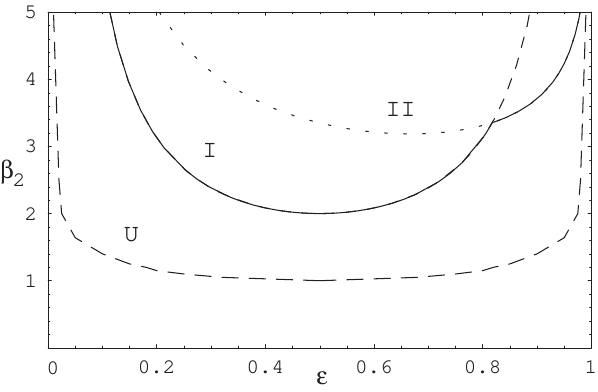}
\end{center}
\caption{An illustration of the lower bound (curve I) given in Proposition \ref{TwoStar}, 
the lower bound (curve II) given in Proposition \ref{VE} and the $U$-shaped region
for the two-star model.}
\end{figure}

\subsection{Along $\epsilon=1/2$ Line}

In general, we proved that there exists some critical number $\beta_{2}^{c}>0$ 
such that the optimizing graphon is uniform for any $0<\beta_{2}<\beta_{2}^{c}$ and non-uniform
for any $\beta_{2}>\beta_{2}^{c}$. But we are far from determining the exact value
of $\beta_{2}^{c}$. For very special case, the two-star model along the line $\epsilon=1/2$,
we can show that $\beta_{2}^{c}=2$. 

\begin{proposition}\label{OneHalfLine}
For the two-star model, along the line $\epsilon=1/2$, the optimizing graphon in \eqref{MainEqn} 
is uniform if $\beta_{2}\leq 2$
and it is not if $\beta_{2}>2$.
\end{proposition}

\begin{proof}
First, by Proposition \ref{TwoStar}, for any $\beta_{2}>\frac{1}{2\epsilon(1-\epsilon)}=\frac{1}{2\frac{1}{2}(1-\frac{1}{2})}=2$,
the optimizer is not uniform.
Next, let us prove that it is uniform if $\beta_{2}\leq 2$.
Let us recall from the proof of Proposition \ref{Multi} that for optimal $h$,
the values of $g(x)=\int_{0}^{1}h(x,y)dy$ are the roots of
\begin{equation}
F(z)=z-\int_{0}^{1}\frac{dy}{1+e^{2\beta_{1}-2\beta_{2}z-2\beta_{2}g(y)}}.
\end{equation}
Differentiating with respect to $z$, we get
\begin{equation}
F'(z)=1-\int_{0}^{1}\frac{2\beta_{2}e^{2\beta_{1}-2\beta_{2}z-2\beta_{2}g(y)}dy}
{(1+e^{2\beta_{1}-2\beta_{2}z-2\beta_{2}g(y)})^{2}}.
\end{equation}
It is clear $F'(z)\geq 1>0$ if $\beta_{2}\leq 0$. 
Now, if $\beta_{2}>0$, 
since $\frac{4x}{(1+x)^{2}}\leq 1$ for any $x>0$, we have
\begin{equation}
F'(z)\geq 1-\frac{\beta_{2}}{2}\geq 0,
\end{equation}
if $\beta_{2}\leq 2$. Suppose $F'(z)=0$, then
the equality holds and we must have $2\beta_{1}-2\beta_{2}z-2\beta_{2}g(y)=0$
for a.e. $y$. Since $\beta_{2}>0$, $g(y)$ is a constant a.e. and so is $h(x,y)$.
Otherwise, we have $F'(z)>0$. When $F$ is strictly increasing, 
$g(x)=\int_{0}^{1}h(x,y)dy$ takes only one value, which is $1/2$ and so 
is $h(x,y)$ for any $x,y$. Thus when $\beta_{2}\leq 2$, the optimizer is uniform.
\end{proof}

\begin{remark}
For general $p$-star model, $p\geq 2$, we can compute that for $\beta_{2}>0$,
\begin{equation}
F'(z)=1-p(p-1)\beta_{2}z^{p-2}\int_{0}^{1}\frac{e^{2\beta_{1}-p\beta_{2}z^{p-1}-p\beta_{2}g^{p-1}(y)}dy}
{(1+e^{2\beta_{1}-p\beta_{2}z^{p-1}-p\beta_{2}g^{p-1}(y)})^{2}}
\leq 1-\frac{p(p-1)}{4}\beta_{2}.
\end{equation}
Thus $F'(z)\geq 0$ if $\beta_{2}\leq\frac{4}{p(p-1)}$. Similarly to the arguments in Proposition \ref{OneHalfLine}, 
we conclude that the optimizing graphon is uniform if $\beta_{2}\leq\frac{4}{p(p-1)}$. Recall
that we already proved that the optimizing graphon is uniform if $(\epsilon,\frac{\beta_{2}}{2})$ is outside
of the $U$-shaped region and in particular when $\epsilon=\frac{p-1}{p}$, the optimizing graphon
is uniform if $\beta_{2}\leq\frac{1}{2}\frac{p^{p-1}}{(p-1)^{p}}$. It is easy to check that
$\frac{4}{p(p-1)}>\frac{1}{2}\frac{p^{p-1}}{(p-1)^{p}}$ and thus provides a better bound
for $p\leq 3$.
\end{remark}

\begin{proposition}
For the two-star model, along the line $\epsilon=1/2$, if $h$ is an optimizer in \eqref{MainEqn} then so is $1-h$.
\end{proposition}

\begin{proof}
It is easy to check that $I(x)=I(1-x)$, $0\leq x\leq 1$ and moreover
\begin{align}
&\int_{0}^{1}\left(\int_{0}^{1}(1-h(x,y))dy\right)^{2}dx
\\
&=\int_{0}^{1}\left(\int_{0}^{1}h(x,y)dy\right)^{2}dx+1-2\int_{0}^{1}\int_{0}^{1}h(x,y)dxdy
\nonumber
\\
&=\int_{0}^{1}\left(\int_{0}^{1}h(x,y)dy\right)^{2}dx
\nonumber
\end{align}
if $\iint_{[0,1]^{2}}h(x,y)dxdy=\frac{1}{2}$. Therefore, if $h$ is an optimizer for the two-star model, so is $1-h$.
\end{proof}

\begin{proposition}\label{Local}
For the two-star model, along the line $\epsilon=1/2$, $\beta_{2}>2$, the graphon
\begin{equation}\label{hDefn}
h(x,y)=
\begin{cases}
\frac{1}{2}+\delta(\beta_{2}) &\text{if $0<x,y<\frac{1}{2}$}
\\
\frac{1}{2}-\delta(\beta_{2}) &\text{if $\frac{1}{2}<x,y<1$}
\\
\frac{1}{2} &\text{otherwise}
\end{cases}
\end{equation}
is a stationary point in \eqref{MainEqn}, where $\delta(\beta_{2})$ is the unique solution to
the equation $\log\left(\frac{\frac{1}{2}+\delta}{\frac{1}{2}-\delta}\right)-2\beta_{2}\delta=0$
on the interval $(0,\frac{1}{2})$.
\end{proposition}

\begin{proof}
Let us consider the graphon
\begin{equation}
h(x,y)=
\begin{cases}
\frac{1}{2}+\delta &\text{if $0<x,y<\frac{1}{2}$}
\\
\frac{1}{2}-\delta &\text{if $\frac{1}{2}<x,y<1$}
\\
\frac{1}{2} &\text{otherwise}
\end{cases},
\end{equation}
where $0\leq\delta<\frac{1}{2}$ is a parameter to be determined later.
It is easy to check that $\iint_{[0,1]^{2}}h(x,y)dxdy=\frac{1}{2}$ and
\begin{equation}
g(x):=\int_{0}^{1}h(x,y)dy=
\begin{cases}
\frac{1}{2}+\frac{\delta}{2} &\text{if $0<x<\frac{1}{2}$}
\\
\frac{1}{2}-\frac{\delta}{2} &\text{if $\frac{1}{2}<x<1$}
\end{cases}.
\end{equation}
Therefore, we have
\begin{equation}
2\beta_{2}-2\beta_{2}g(x)-2\beta_{2}g(y)=\log\left(\frac{1-h(x,y)}{h(x,y)}\right),
\qquad
0<x,y<1,
\end{equation}
if we let $2\beta_{2}\delta=\log\left(\frac{\frac{1}{2}+\delta}{\frac{1}{2}-\delta}\right)$.
Hence, the graphon satisfies the Euler-Lagrange equation and is therefore a stationary point.
For any $\beta_{2}>2$, let us define
\begin{equation}
G(\delta):=\log\left(\frac{\frac{1}{2}+\delta}{\frac{1}{2}-\delta}\right)-2\beta_{2}\delta.
\end{equation}
Then, $G(0)=0$, $\lim_{\delta\uparrow\frac{1}{2}}G(\delta)=+\infty$ and
\begin{equation}
G'(\delta)=\frac{1}{\frac{1}{2}-\delta}+\frac{1}{\frac{1}{2}+\delta}-2\beta_{2},
\qquad
G''(\delta)=\frac{1}{(\frac{1}{2}-\delta)^{2}}-\frac{1}{(\frac{1}{2}+\delta)^{2}}.
\end{equation}
Thus, $G''(\delta)>0$ for any $0<\delta<\frac{1}{2}$ and $G'(0)<0$ since $\beta_{2}>2$. Therefore,
$G(\delta)=0$ has a unique solution on $(0,\frac{1}{2})$.
\end{proof}

\begin{remark}
It would be interesting to know if the $h$ defined in \eqref{hDefn} is indeed the optimizer. 
That does not seem to be the case. Indeed, one can show that
the $h$ defined in \eqref{hDefn} in Proposition \ref{Local} is a saddle point at least for $\beta_{2}>4$.
Up to second variation, $\iint_{[0,1]^{2}}\delta h(x,y)dxdy=0$
and $\delta h(x,y)=\delta h(y,x)$,
\begin{align}\label{UpperBoundI}
\delta\psi
&=\beta_{2}\int_{0}^{1}\left(\int_{0}^{1}h(x,y)+\delta h(x,y)dy\right)^{2}dx
-\beta_{2}\int_{0}^{1}\left(\int_{0}^{1}h(x,y)dy\right)^{2}dx
\\
&\qquad\qquad\qquad\qquad
-\frac{1}{2}\iint_{[0,1]^{2}}[I(h+\delta h)-I(h)]dxdy
\nonumber
\\
&=2\beta_{2}\int_{0}^{1}\left(\int_{0}^{1}h(x,y)dy\right)\left(\int_{0}^{1}\delta h(x,y)dy\right)dx
+\beta_{2}\int_{0}^{1}\left(\int_{0}^{1}\delta h(x,y)dy\right)^{2}dx
\nonumber
\\
&\qquad\qquad
-\frac{1}{2}\iint_{[0,1]^{2}}I'(h)\delta hdxdy
-\frac{1}{4}\iint_{[0,1]^{2}}I''(h)(\delta h)^{2}dxdy
\nonumber
\\
&=
\beta_{2}\int_{0}^{1}\left(\int_{0}^{1}\delta h(x,y)dy\right)^{2}dx
-\frac{1}{4}\iint_{[0,1]^{2}}I''(h)(\delta h)^{2}dxdy
\nonumber
\end{align}

Moreover, observe that
\begin{equation}\label{ISecond}
I''(h)=
\begin{cases}
\frac{1}{\frac{1}{2}-\delta}+\frac{1}{\frac{1}{2}+\delta} &\text{if $0<x,y<\frac{1}{2}$ or $\frac{1}{2}<x,y<1$}
\\
4 &\text{otherwise}
\end{cases}.
\end{equation}

Therefore, by \eqref{UpperBoundI} and \eqref{ISecond}, we have
\begin{align}\label{UpperBoundIV}
\delta\psi
&=\beta_{2}\int_{0}^{1}\left(\int_{0}^{1}\delta h(x,y)dy\right)^{2}dx
\\
&\qquad\qquad
-\frac{1}{1-4\delta^{2}}\iint_{R_{1}}(\delta h(x,y))^{2}dxdy
-\iint_{R_{2}}(\delta h(x,y))^{2}dxdy
\nonumber
\end{align}
where
\begin{equation}
R_{1}:=\left\{(x,y):0<x,y<\frac{1}{2}\right\}\bigcup\left\{(x,y):\frac{1}{2}<x,y<1\right\},
\quad
R_{2}:=[0,1]^{2}\backslash R_{1}.
\end{equation}

Consider $\delta h(x,y)$ defined as $\delta h(x,y)=\epsilon$ if $(x,y)\in R_{1}$
and $\delta h(x,y)=-\epsilon$ if $(x,y)\in R_{2}$. Then, $\iint_{[0,1]^{2}}\delta h(x,y)dxdy=0$ and we can compute that
\begin{equation}
\delta\psi=-\frac{1}{2}\left(\frac{1}{1-4\delta^{2}}+1\right)\epsilon^{2}.
\end{equation}
On the other hand, consider $\delta h(x,y)$ defined as
\begin{equation}
\delta h(x,y)=
\begin{cases}
\epsilon &\text{if $0<x<\frac{1}{2}$,$\frac{3}{4}<y<1$ or $\frac{3}{4}<x<1$, $0<y<\frac{1}{2}$}
\\
-\epsilon &\text{if $0<x<\frac{1}{2}$,$\frac{1}{2}<y<\frac{3}{4}$ or $\frac{1}{2}<x<\frac{3}{4}$, $0<y<\frac{1}{2}$}
\\
0 &\text{otherwise}
\end{cases}.
\end{equation}
Then, $\iint_{[0,1]^{2}}\delta h(x,y)dxdy=0$ and we can compute that
\begin{equation}
\delta\psi=\frac{\beta_{2}}{2}\frac{\epsilon^{2}}{4}-\frac{\epsilon^{2}}{2}.
\end{equation}
Hence, for $\beta_{2}>4$, $h$ defined in \eqref{hDefn} is a saddle point.
\end{remark}

\subsection{Phase Transition}

In Proposition \ref{VE}, we showed that for a general simple subgraph $H$ satisfying
the condition $e(H)>v(H)/2$, when
\begin{equation}
\beta_{2}>\frac{-\frac{1}{2}[\epsilon\log\epsilon+(1-\epsilon)\log(1-\epsilon)]}{\epsilon^{v(H)/2}-\epsilon^{e(H)}},
\end{equation}
the optimizing graphon is not uniform. On the other hand, by Proposition \ref{PositiveUniform},
for $\beta_{2}\geq 0$, there exists a $U$-shaped region outside of which the optimizing graph
is uniform. Therefore, fix the edge density $\epsilon$, if we view $\psi(\epsilon,\beta_{2})$ as a function
of $\beta_{2}$, it is constant in $\beta_{2}$ on a non-trivial interval. By complex analysis, if $\psi(\epsilon,\beta_{2})$
were analytic in $\beta_{2}$, then it would be constant everywhere. Hence, for any fixed edge density $\epsilon$, 
there exists a positive $\beta_{2}$ at which we have non-analyticity.
It is also worth mentioning that the non-analyticity in positive $\beta_{2}$ may
be alternatively derived using Theorem 1.1. in \cite{RadinIV}. This is also briefly mentioned in \cite{KenyonYin},
where the non-analyticity in negative $\beta_{2}$ is proved.

\section{Summary and Open Questions}\label{SummarySection}

We have studied the constrained exponential random graph models introduced by Kenyon and Yin \cite{KenyonYin}. We showed
uniform and non-uniform structure for very general underlying graph $H$. 
For $\beta_{2}$ close to zero, either in the attractive regime or repulsive regime, 
the optimal graphon will be uniform and for $\beta_{2}$ sufficiently large, the optimal graphon
will be non-uniform. It remains open to find the exact optimizing graphon structure for
finite large $\beta_{2}$. It is worth mentioning that similar phenomena have been observed
for the micro-canonical ensembles in Radin and Sadun \cite{RadinIV}. They showed
how the entropy changes when it is close to the so-called Erd\H{o}s-Reny\'{i} density. That can give an alternative approach
to giving some estimates on when the asymptotic structure for the canonical ensembles is uniform that was
considered in this paper.

More results are obtained when $H$ is a $p$-star. In the case when $H$ is a two-star,
we can show that along the line $\epsilon=1/2$, the asymptotic structure is uniform
if $\beta_{2}\leq 2$ and is non-uniform if $\beta_{2}>2$. For general $H$, we do not have a sharp result.
This remains the major challenging open problem for future investigations.
Even if we cannot get a sharp result for general $H$, is it possible to show a sharp transition
for a concrete model, e.g. edge-triangle model along the line $\epsilon=1/2$?

We also found and proved a stationary point for the two-star model and it remains an open question
if it is indeed a local/global optimizer. Similar results should hold for the corresponding
micro-canonical model. 
When $H$ is a $p$-star, we showed
that the optimizing graphon must be multipodal. The numerical results for the corresponding
micro-canonical model suggest that the optimizing graphons should indeed be bipodal, see Kenyon et al \cite{Kenyon}.
The same conjecture can be said in our case.

\section*{Acknowledgements}

The author is very grateful to two anonymous referees and the editor
for helpful comments and suggestions. The author also thanks David Aristoff for helpful discussions.
The author is partially supported by NSF Grant DMS-1613164.

\end{document}